\documentclass{amsart}
\usepackage{amsmath,amssymb}
\usepackage[dvips]{graphics}
\usepackage[dvipdfmx]{graphicx}
\usepackage[all]{xy}

\newtheorem{Theorem}{Theorem}[section]
\newtheorem{Lemma}[Theorem]{Lemma}

\newtheorem{Corollary}[Theorem]{Corollary}

\theoremstyle{definition}

\theoremstyle{remark}
\newtheorem{Remark}[Theorem]{Remark}

\def\labelenumi{\rm ({\makebox[0.8em][c]{\theenumi}})}
\def\theenumi{\arabic{enumi}}
\def\labelenumii{\rm {\makebox[0.8em][c]{(\theenumii)}}}
\def\theenumii{\roman{enumii}}

\renewcommand{\thefigure}
{\arabic{section}.\arabic{figure}}
\makeatletter
\@addtoreset{figure}{section}
\def\@thmcountersep{-}
\makeatother

\numberwithin{equation}{section}

\begin{document} 

\title[A short proof of the generalized Conway--Gordon--Sachs theorem]{A short proof of the generalized Conway--Gordon--Sachs theorem}

\author{Ryo Nikkuni}
\address{Department of Information and Mathematical Sciences, School of Arts and Sciences, Tokyo Woman's Christian University, 2-6-1 Zempukuji, Suginami-ku, Tokyo 167-8585, Japan}
\email{nick@lab.twcu.ac.jp}
\thanks{The author was supported by JSPS KAKENHI Grant Number JP19K03500.}


\date{}



\begin{abstract}
The famous Conway--Gordon--Sachs theorem for the complete graph on six vertices was extended to the general complete graph on $n$ vertices by Kazakov--Korablev as a congruence modulo $2$, and its integral lift was given by Morishita--Nikkuni. However, the proof is complicated and long. In this paper, we provide a shorter proof of the generalized Conway--Gordon--Sachs theorem over integers.
\end{abstract}

\maketitle

\section{Introduction} 

An embedding $f$ of a finite graph $G$ into the $3$-sphere is called a {\it spatial embedding} of $G$, and $f(G)$ is called a {\it spatial graph} of $G$. We call a subgraph of $G$ homeomorphic to the circle a {\it cycle} of $G$, and a cycle containing exactly $p$ vertices a {\it $p$-cycle}. We denote the set of all $p$-cycles of $G$ by $\Gamma_{p}(G)$. Additionally, we also denote the set of all disjoint pairs of cycles of $G$ consisting of a $p$-cycle and a $q$-cycle by $\Gamma_{p,q}(G)$. In general, for a closed $1$-manifold $\lambda$ in $G$, the image $f(\lambda)$ is a knot or a link contained in $f(G)$. In particular, if $\lambda$ belongs to $\Gamma_{p,q}(G)$, then we call $f(\lambda)$ a {\it $(p,q)$-link} of $f(G)$. Moreover, if $\lambda$ contains all vertices of $G$, then we also call $f(\lambda)$ a {\it Hamiltonian link} of $f(G)$.

Let $K_{n}$ be the {\it complete graph} on $n$ vertices, that is the graph consisting of $n$ vertices such that each pair of two distinct vertices is connected by exactly one edge. The fact that for every spatial graph of $K_6$, the sum of the linking numbers over all of the $(3,3)$-links is odd is well known as the {\it Conway-Gordon--Sachs Theorem} \cite{CG83}, \cite{S84}. If $n$ is 7 or more, Kazakov--Korablev showed that for every spatial graph of $K_{n}$, the sum of the linking numbers over all of the Hamiltonian links is even \cite{KK14}. On the other hand, Morishita and the author significantly generalized this modulo $2$ congruence formula over integers as follows.

\begin{Theorem}{\rm (Morishita--Nikkuni \cite{MN21})}\label{lkrefine} 
Let $n\ge 6$ be an integer and $p,q\ge 3$ two integers satisfying $n=p+q$. For every spatial graph $f(K_n)$ of $K_{n}$, the following holds: 
\begin{eqnarray}\label{ilKK0}
\sum_{\lambda\in\Gamma_{p,q}(K_n)}{\rm lk}(f(\lambda))^2
=(2-\delta_{pq})\cdot(n-6)!\sum_{\lambda\in\Gamma_{3,3}(K_n)}{\rm lk}(f(\lambda))^2. 
\end{eqnarray}
Here, {\rm lk} denotes the linking number and $\delta_{pq}$ denotes the Kronecker's delta. 
\end{Theorem}

It also holds from (\ref{ilKK0}) that 
\begin{eqnarray}\label{ilKK}
\sum_{p+q=n}
\sum_{\lambda\in\Gamma_{p,q}(K_n)}{\rm lk}(f(\lambda))^2
=
(n-5)!\sum_{\lambda\in\Gamma_{3,3}(K_n)}{\rm lk}(f(\lambda))^2.
\end{eqnarray}
Kazakov--Korablev's result above can be obtained immediately by taking the modulo $2$ reduction on both sides of (\ref{ilKK}). Since the right side of (\ref{ilKK}) is not zero, it is also clear that for any $p$ and $q$ satisfying $p+q=n$, there always exists a nonsplitable Hamiltonian $(p,q)$-link of $f(K_{n})$. This is also a generalization of Vesnin--Litvintseva's result \cite{YL10} that $f(K_{n})$ always contain a nonsplittable Hamiltonian link. Furthermore, it has also been shown in \cite{MN21} that for $n\ge 6$, the following congruence modulo $2\cdot (n-5)!$ can be obtained from (\ref{ilKK}): 
\begin{eqnarray*}\label{congru3}
\sum_{p+q=n}
\sum_{\lambda\in\Gamma_{p,q}(K_n)}\!\!\!\!{\rm lk}(f(\lambda))^2 
\equiv 
\left\{
   \begin{array}{@{\,}lll}
   (n-5)! & (n\equiv 6,7\pmod{8})\\
   0 & (n\not\equiv 6,7\pmod{8}). 
   \end{array}
\right.
\end{eqnarray*}
For additional formulas on evaluating the sum of linking numbers over Hamiltonian links derived from (\ref{ilKK0}), see \cite{MN21}. The proof of Theorem \ref{lkrefine} given in \cite{MN21} was done using tricky induction procedures and was lengthy and tedious. Our purpose in this paper is to simplify and shorten the proof of Theorem \ref{lkrefine}, thereby making it more accessible to a wider audience. The key to this simplification is Lemma \ref{circleKnlemma}, which reveals a particular homological property concerning the linking of a knot with a spatial complete graph. The proof of Theorem \ref{lkrefine} presented here is a reorganized version of the one originally introduced by the author in the Japanese textbook \cite{Nikkuni22} several years ago.

\begin{Remark}
Conway--Gordon also showed that for every spatial graph of $K_{7}$, the sum of the Arf invariants over all of the Hamiltonian knots is odd \cite{CG83}. For further details on refining this fact or extending it to general $K_{n}$, see \cite{Nikkuni09}, \cite{MN19}. 
\end{Remark}

\section{New proof of Theorem \ref{lkrefine}} 

In the following, let $K_{n}$ consist of $n$ vertices $1,2,\ldots,n$. We denote the edge of $K_n$ connecting two distinct vertices $i$ and $j$ by $\overline{ij}$, and denote a path of length 2 of $K_n$ consisting of two edges $\overline{ij}$ and $\overline{jk}$ by $\overline{ijk}$.

\begin{Lemma}\label{circleK4lemma} 
Let $G$ be the graph which is a disjoint union of a loop $e$ and $K_{4}$. For every spatial graph $f(G)$ of $G$, the following holds: 
\begin{eqnarray}\label{cK4lk} 
\sum_{\delta\in\Gamma_{4}(K_{4})}{\rm lk}(f(e),f(\delta))^2
=\sum_{\gamma\in\Gamma_{3}(K_{4})}{\rm lk}(f(e),f(\gamma))^2. 
\end{eqnarray}
\end{Lemma}

\begin{proof}
We give the orientation to the loop $e$ and each of the edges of $K_{4}$ as shown in Fig. \ref{circleK4}, and identify each of the cycles of $K_{4}$ with an element of the first integral homology group $H_{1}(K_{4};{\mathbb Z})$ as follows: 
\begin{eqnarray*}
&&\gamma_{1} = \langle 12\rangle + \langle 24\rangle - \langle 14\rangle,\ 
\gamma_{2} = \langle 23\rangle + \langle 34\rangle - \langle 24\rangle,\
\gamma_{3} = \langle 31\rangle + \langle 14\rangle - \langle 34\rangle,\\
&&\gamma_{4} = \langle 12\rangle + \langle 23\rangle + \langle 31\rangle,\
\delta_{1} = \langle 23\rangle + \langle 31\rangle + \langle 14\rangle - \langle 24\rangle,\\
&&\delta_{2} = \langle 31\rangle + \langle 12\rangle + \langle 24\rangle - \langle 34\rangle,\ \delta_{3} = \langle 12\rangle + \langle 23\rangle + \langle 34\rangle - \langle 14\rangle. 
\end{eqnarray*}
Note that $\gamma_{i}\in \Gamma_{3}(K_{4})\ (i=1,2,3,4)$ and $\delta_{j}\in \Gamma_{4}(K_{4})\ (j=1,2,3)$. Since $\delta_{j}=\gamma_{j+1}+\gamma_{j+2}$ in $H_{1}(K_{4};{\mathbb Z})$, we have 
\begin{eqnarray*}
{\rm lk}(f(e),f(\delta_{j})) &=& {\rm lk}(f(e),f(\gamma_{j+1}))+{\rm lk}(f(e),f(\gamma_{j+2})), 
\end{eqnarray*}
where we regard $\gamma_{j+3}=\gamma_{j}$. Then we have 
\begin{eqnarray}\label{ck41}
\sum_{j=1}^{3}{\rm lk}(f(e),f(\delta_{j}))^{2}
&=& 2\sum_{i=1}^{3}{\rm lk}(f(e),f(\gamma_{i}))^{2}\\
&& + 2\sum_{1\le i<j\le 3}{\rm lk}(f(e),f(\gamma_{i})){\rm lk}(f(e),f(\gamma_{j})). \nonumber
\end{eqnarray}
On the other hand, since $\gamma_{4}=\gamma_{1}+\gamma_{2}+\gamma_{3}$ in $H_{1}(K_{4};{\mathbb Z})$, we have 
\begin{eqnarray}\label{ck42}
{\rm lk}(f(e),f(\gamma_{4})) = \sum_{i=1}^{3}{\rm lk}(f(e),f(\gamma_{i})).
\end{eqnarray}
By (\ref{ck41}) and (\ref{ck42}), we have 
\begin{eqnarray*}
\sum_{i=1}^{4}{\rm lk}(f(e),f(\gamma_{i}))^{2}
&=& \sum_{i=1}^{3}{\rm lk}(f(e),f(\gamma_{i}))^{2} + {\rm lk}(f(e),f(\gamma_{4}))^{2}\\
&=& 2\sum_{i=1}^{3}{\rm lk}(f(e),f(\gamma_{i}))^{2}\\
&& + 2\sum_{1\le i<j\le 3}{\rm lk}(f(e),f(\gamma_{i})){\rm lk}(f(e),f(\gamma_{j}))\\
&=& \sum_{j=1}^{3}{\rm lk}(f(e),f(\delta_{j}))^{2}. 
\end{eqnarray*}
This completes the proof. 
\end{proof}

\begin{figure}[htbp]
\begin{center}
\scalebox{0.6}{\includegraphics*{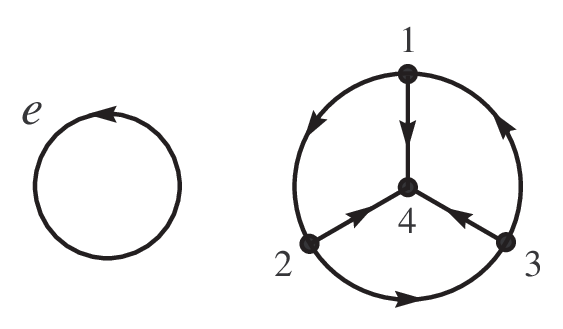}}
\caption{Oriented loop $e$ and oriented $K_{4}$}
\label{circleK4}
\end{center}
\end{figure}

\begin{Lemma}\label{circleKnlemma} 
Let $n\ge 4$ be an integer and $G$ the graph which is a disjoint union of a loop $e$ and $K_{n}$. For every spatial graph $f(G)$ of $G$, the following holds: 
\begin{eqnarray}\label{cKnlk} 
\sum_{\gamma\in\Gamma_{n}(K_{n})}{\rm lk}(f(e),f(\gamma))^2
=\sum_{\gamma\in\Gamma_{n-1}(K_{n})}{\rm lk}(f(e),f(\gamma))^2. 
\end{eqnarray}
\end{Lemma}

\begin{proof}
Let us prove (\ref{cKnlk}) by induction on $n$. The case where $n=4$ is nothing other than Lemma \ref{circleK4lemma}. So we assume that $n>4$. Let $f(G)$ be a spatial graph of $G$. For a vertex $m$ of $K_{n}$ $(m=1,2,\ldots,n)$, we denote the subgraph obtained by removing $m$ and all edges incident to $m$ from $K_{n}$ by $K_{n-1}^{(m)}$. This subgraph is isomorphic to $K_{n-1}$, see the left figure in Fig. \ref{Knsubdivide}. Furthermore, for vertices $i,j$ and $m$ of $K_{n}$ such that $1\le i<j\le n$ and $i,j\neq m$, let $F_{ij}^{(m)}$ denote the subgraph obtained by removing the edge $\overline{ij}$ and $n-3$ edges $\overline{mk}$ $(1\le k\le n,\ k\neq i,j)$. This subgraph is not isomorphic to $K_{n-1}$ but is homeomorphic to it. Actually, $F_{ij}^{(m)}$ is obtained by replacing the edge $\overline{ij}$ in $K_{n-1}^{(m)}$ with the path $\overline{imj}$ of length $2$, see the right figure in Fig. \ref{Knsubdivide}. Then for the spatial subgraph $f(e\cup F_{ij}^{(m)})$, by the assumption we have 
\begin{eqnarray}\label{ind_hyp}
&&\sum_{\gamma\in\Gamma_{n}(F_{ij}^{(m)})}{\rm lk}(f(e),f(\gamma))^{2} 
+\sum_{\substack{\gamma\in\Gamma_{n-1}(K_{n-1}^{(m)}) \\ \overline{ij}\not\subset \gamma}}{\rm lk}(f(e),f(\gamma))^{2}\\
&=& \sum_{\substack{\gamma\in\Gamma_{n-1}(F_{ij}^{(m)}) \\ \overline{imj}\subset \gamma}} {\rm lk}(f(e),f(\gamma))^{2}
+\sum_{\substack{\gamma\in\Gamma_{n-2}(K_{n-1}^{(m)}) \\ \overline{ij}\not\subset \gamma}}{\rm lk}(f(e),f(\gamma))^{2}. \nonumber
\end{eqnarray}
Let us fix the vertex $m$ and add both sides of (\ref{ind_hyp}) for all $i,j$ such that $1\le i<j\le n$ and $i,j\neq m$. First, consider an $n$-cycle $\gamma$ of $K_{n}$ where the vertices $i$ and $j$ are adjacent to the vertex $m$ on $\gamma$. In this case, $\gamma$ is an $n$-cycle of $F_{ij}^{(m)}$ and we have 
\begin{eqnarray}\label{eq2l}
\sum_{\substack{1\le i<j\le n \\ i,j\neq m}}\bigg(\sum_{\gamma\in\Gamma_{n}(F_{ij}^{(m)})}{\rm lk}(f(e),f(\gamma))^{2}\bigg)
= \sum_{\gamma\in\Gamma_{n}(K_{n})}{\rm lk}(f(e),f(\gamma))^{2}. 
\end{eqnarray}
Next, consider an $(n-1)$-cycle $\gamma$ of $K_{n-1}^{(m)}$. Let $\overline{ij}$ be an edge of $K_{n-1}^{(m)}$ that is not contained in $\gamma$. Then there are $\binom{n - 1}{2} - (n-1) = (n^{2} -5n +4)/2$ ways to choose such $i$ and $j$, and thus we have 
\begin{eqnarray}\label{eq3l}
&& \sum_{\substack{1\le i<j\le n \\ i,j\neq m}}\bigg(\sum_{\substack{\gamma\in\Gamma_{n-1}(K_{n-1}^{(m)}) \\ \overline{ij}\not\subset \gamma}}{\rm lk}(f(e),f(\gamma))^{2}\bigg) \\ 
&=& \frac{n^{2} -5n +4}{2}\sum_{\gamma\in\Gamma_{n-1}(K_{n-1}^{(m)})}{\rm lk}(f(e),f(\gamma))^{2}. \nonumber
\end{eqnarray}
Next, consider an $(n-1)$-cycle $\gamma$ of $K_{n}$ containing the vertex $m$, where the vertices $i$ and $j$ are adjacent to $m$ on $\gamma$. Then $\gamma$ is an $(n-1)$-cycle of $F_{ij}^{(m)}$ containing the path $\overline{imj}$ and we have 
\begin{eqnarray}\label{eq4l}
\sum_{\substack{1\le i<j\le n \\ i,j\neq m}}\bigg(\sum_{\substack{\gamma\in\Gamma_{n-1}(F_{ij}^{(m)}) \\ \overline{imj}\subset \gamma}}\!\!\!\! {\rm lk}(f(e),f(\gamma))^{2}\bigg)
&=& \sum_{\substack{\gamma\in\Gamma_{n-1}(K_{n}) \\ m\subset \gamma}}\!\!\!\! {\rm lk}(f(e),f(\gamma))^{2}. 
\end{eqnarray}
Finally, consider an $(n-2)$-cycle $\gamma$ of $K_{n-1}^{(m)}$. Let $\overline{ij}$ be an edge of $K_{n-1}^{(m)}$ that is not contained in $\gamma$. Then there are $\binom{n - 1}{2} - (n-2) = (n^{2} - 5n +6)/2$ ways to choose such $i$ and $j$, and thus we have 
\begin{eqnarray}\label{eq5l}
&& \sum_{\substack{1\le i<j\le n \\ i,j\neq m}}\bigg(\sum_{\substack{\gamma\in\Gamma_{n-2}(K_{n-1}^{(m)}) \\ \overline{ij}\not\subset \gamma}}{\rm lk}(f(e),f(\gamma))^{2}\bigg)\\
&=& \frac{n^{2} - 5n +6}{2}\sum_{\gamma\in\Gamma_{n-2}(K_{n-1}^{(m)})}{\rm lk}(f(e),f(\gamma))^{2}. \nonumber
\end{eqnarray}
Furthermore, for each $f(K_{n-1}^{(m)})$, the following holds by the assumption:
\begin{eqnarray}\label{eq9l}
\sum_{\gamma\in \Gamma_{n-1}(K_{n-1}^{(m)})}{\rm lk}(f(e),f(\gamma))^{2}
= \sum_{\gamma\in \Gamma_{n-2}(K_{n-1}^{(m)})}{\rm lk}(f(e),f(\gamma))^{2}. 
\end{eqnarray}
Theorefore, from (\ref{ind_hyp}) and (\ref{eq2l}), (\ref{eq3l}), (\ref{eq4l}), (\ref{eq5l}), (\ref{eq9l}), we have 
\begin{eqnarray}\label{eq10l}
&& \sum_{\gamma\in \Gamma_{n}(K_{n})}{\rm lk}(f(e),f(\gamma))^{2}\\
&=& \sum_{\substack{\gamma\in\Gamma_{n-1}(K_{n}) \\ m\subset \gamma}}{\rm lk}(f(e),f(\gamma))^{2} + \sum_{\gamma\in \Gamma_{n-1}(K_{n-1}^{(m)})}{\rm lk}(f(e),f(\gamma))^{2}. \nonumber
\end{eqnarray}
Then, let us add both sides of (\ref{eq10l}) for all $m = 1,2,\ldots,n$. First, in the $(n-1)$-cycle $\gamma$ of $K_{n}$, there are $(n-1)$ ways to choose the vertex $m$ that is contained in $\gamma$. Thus we have 
\begin{eqnarray}\label{eq11l}
\sum_{m=1}^{n}\bigg(\sum_{\substack{\gamma\in\Gamma_{n-1}(K_{n}) \\ m\subset \gamma}}\!\!\!\! {\rm lk}(f(e),f(\gamma))^{2}\bigg) = (n-1)\sum_{\gamma\in \Gamma_{n-1}(K_{n})}\!\!\!\! {\rm lk}(f(e),f(\gamma))^{2}. 
\end{eqnarray}
Second, in the $(n-1)$-cycle $\gamma$ of $K_{n}$, for the vertex $m$ that is not contained in $\gamma$, the cycle $\gamma$ is an $(n-1)$-cycle of $K_{n-1}^{(m)}$. Since there is only one way to choose such an $m$, we have 
\begin{eqnarray}\label{eq12l}
\sum_{m=1}^{n}\bigg(\sum_{\gamma\in\Gamma_{n-1}(K_{n-1}^{(m)})}{\rm lk}(f(e),f(\gamma))^{2}\bigg)
= \sum_{\gamma\in \Gamma_{n-1}(K_{n})}{\rm lk}(f(e),f(\gamma))^{2}. 
\end{eqnarray}
Then by substituting (\ref{eq11l}) and (\ref{eq12l}) into (\ref{eq10l}), we have 
\begin{eqnarray*}
n\sum_{\gamma\in \Gamma_{n}(K_{n})}{\rm lk}(f(e),f(\gamma))^{2} 
= n \sum_{\gamma\in \Gamma_{n-1}(K_{n})}{\rm lk}(f(e),f(\gamma))^{2}. 
\end{eqnarray*}
This completes the proof. 
\end{proof}

\begin{figure}[htbp]
\begin{center}
\scalebox{0.475}{\includegraphics*{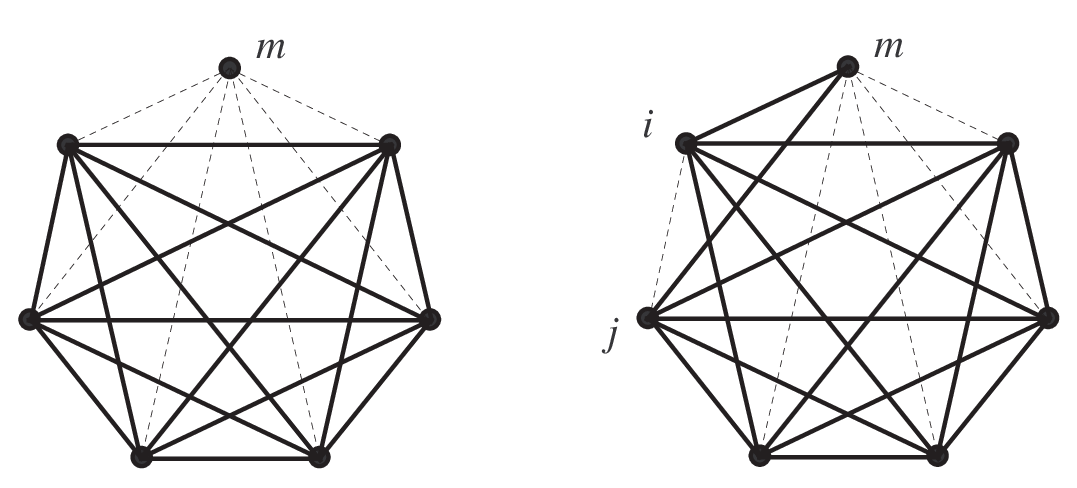}}
\caption{$K_{n-1}^{(m)}$ (left) and $F_{ij}^{(m)}$ (right) $(n=7)$}
\label{Knsubdivide}
\end{center}
\end{figure}

\begin{Corollary}\label{circleKncor} 
Let $n\ge 4$ be an integer and $G$ the graph which is a disjoint union of a loop $e$ and $K_{n}$. For every spatial graph $f(G)$ of $G$, the following holds: 
\begin{eqnarray}\label{cKnlk3} 
\sum_{\gamma\in\Gamma_{n}(K_{n})}{\rm lk}(f(e),f(\gamma))^2
=(n-3)!\sum_{\gamma\in\Gamma_{3}(K_{n})}{\rm lk}(f(e),f(\gamma))^2. 
\end{eqnarray}
\end{Corollary}

\begin{proof}
We denote the subgraph of $K_{n}$ obtained by removing exactly $k$ vertices $m_{1},m_{2},\ldots,m_{k}\ (1\le k\le n-1)$ and all edges incident to them from $K_{n}$ by $K_{n-k}^{(m_{1}m_{2}\cdots m_{k})}$. Note that this subgraph is isomorphic to $K_{n-k}$. Then by repeatedly using Lemma \ref{circleKnlemma}, we have 
\begin{eqnarray*}
&& \sum_{\gamma\in \Gamma_{n}(K_{n})}{\rm lk}(f(e),f(\gamma))^{2} 
= \sum_{\gamma\in \Gamma_{n-1}(K_{n})}{\rm lk}(f(e),f(\gamma))^{2}\\ 
&=& \sum_{m_{1}=1}^{n}\bigg(\sum_{\gamma\in \Gamma_{n-1}(K_{n-1}^{(m_{1})})}{\rm lk}(f(e),f(\gamma))^{2}\bigg)\\
&=& \sum_{m_{1}=1}^{n}\bigg(\sum_{\gamma\in \Gamma_{n-2}(K_{n-1}^{(m_{1})})}{\rm lk}(f(e),f(\gamma))^{2}\bigg)
= 
2\sum_{\gamma\in \Gamma_{n-2}(K_{n})}{\rm lk}(f(e),f(\gamma))^{2}\\
&=& 2\sum_{1\le m_{1}<m_{2}\le n}\bigg(\sum_{\gamma\in \Gamma_{n-2}(K_{n-2}^{(m_{1}m_{2})})}{\rm lk}(f(e),f(\gamma))^{2}\bigg)\\
&=& 2\sum_{1\le m_{1}<m_{2}\le n}\bigg(\sum_{\gamma\in \Gamma_{n-3}(K_{n-2}^{(m_{1}m_{2})})}{\rm lk}(f(e),f(\gamma))^{2}\bigg)\\
&=& 2\cdot 3\sum_{\gamma\in \Gamma_{n-3}(K_{n})}{\rm lk}(f(e),f(\gamma))^{2}
= \cdots = k! \sum_{\gamma\in \Gamma_{n-k}(K_{n})}{\rm lk}(f(e),f(\gamma))^{2}.
\end{eqnarray*} 
So if we set $k=n-3$, we have the desired conclusion. 
\end{proof} 

\begin{proof}[Proof of Theorem \ref{lkrefine}]
For a $k$-cycle $\gamma$ of $K_n$, let $G_{\gamma}$ denote the subgraph of $K_n$ obtained by removing the $k$ vertices of $\gamma$ and their edges. Note that this subgraph is isomorphic to $K_{n-k}$. Then we have 
\begin{eqnarray}\label{lkpqeq}
&&(1+ \delta_{pq})\sum_{\lambda\in \Gamma_{p,q}(K_{n})}{\rm lk}(f(\lambda))^{2}\\
&=& \sum_{\gamma\in \Gamma_{p}(K_{n})}
\bigg(
\sum_{\gamma'\in \Gamma_{q}(G_{\gamma})}{\rm lk}(f(\gamma),f(\gamma'))^{2}
\bigg). \nonumber
\end{eqnarray}
Then by Corollary \ref{circleKncor}, the right side of (\ref{lkpqeq}) can be expressed as follows: 
\begin{eqnarray}\label{lkpqeq2}
&& \sum_{\gamma\in \Gamma_{p}(K_{n})}
\bigg(
\sum_{\gamma'\in \Gamma_{q}(G_{\gamma})}{\rm lk}(f(\gamma),f(\gamma'))^{2}
\bigg) \\
&=& \sum_{\gamma\in \Gamma_{p}(K_{n})}
\bigg(
(q-3)!\sum_{\gamma'\in \Gamma_{3}(G_{\gamma})}{\rm lk}(f(\gamma),f(\gamma'))^{2}
\bigg)\nonumber \\
&=& (q-3)!\sum_{\gamma'\in \Gamma_{3}(K_{n})}
\bigg(
\sum_{\gamma\in \Gamma_{p}(G_{\gamma'})}{\rm lk}(f(\gamma),f(\gamma'))^{2}
\bigg).\nonumber
\end{eqnarray}
Let $H_{\gamma'}^{i}\ (i=1,2,\ldots,l,\ l=\binom{n-3}{p})$ be the collection of all subgraphs of $G_{\gamma'}$ that are isomorphic to $K_{p}$. Then the sets of $p$-cycles $\Gamma_{p}(H_{\gamma'}^{i})$ are mutually disjoint, and their union is $\Gamma_{p}(G_{\gamma'})$. Then by Corollary \ref{circleKncor}, we have 
\begin{eqnarray}\label{lkpqeq3}
&& \sum_{\gamma'\in \Gamma_{3}(K_{n})}
\bigg(
\sum_{\gamma\in \Gamma_{p}(G_{\gamma'})}{\rm lk}(f(\gamma),f(\gamma'))^{2}
\bigg)\\
&=& 
\sum_{\gamma'\in \Gamma_{3}(K_{n})}
\bigg(
\sum_{i=1}^{l}
\bigg(
\sum_{\gamma\in \Gamma_{p}(H_{\gamma'}^{i})}{\rm lk}(f(\gamma),f(\gamma'))^{2}
\bigg)
\bigg)\nonumber\\
&=& 
\sum_{\gamma'\in \Gamma_{3}(K_{n})}
\bigg(
\sum_{i=1}^{l}
\bigg(
(p-3)!\sum_{\gamma\in \Gamma_{3}(H_{\gamma'}^{i})}{\rm lk}(f(\gamma),f(\gamma'))^{2}
\bigg)
\bigg)\nonumber\\
&=& 
(p-3)!\sum_{\gamma'\in \Gamma_{3}(K_{n})}
\bigg(
\sum_{i=1}^{l}
\bigg(
\sum_{\gamma\in \Gamma_{3}(H_{\gamma'}^{i})}{\rm lk}(f(\gamma),f(\gamma'))^{2}
\bigg)
\bigg).\nonumber
\end{eqnarray}
Note that each $3$-cycle $\gamma$ of $G_{\gamma'}$ is shared by exactly $\binom{n-6}{p-3}$ of the $H_{\gamma'}^{i}$ subgraphs. This implies that 
\begin{eqnarray}\label{lkpqeq4}
\sum_{i=1}^{l}
\bigg(
\sum_{\gamma\in \Gamma_{3}(H_{\gamma'}^{i})}\!\!\!\! {\rm lk}(f(\gamma),f(\gamma'))^{2}
\bigg)
= \binom{n-6}{p-3}
\sum_{\gamma\in \Gamma_{3}(G_{\gamma'})}\!\!\!\! {\rm lk}(f(\gamma),f(\gamma'))^{2}. 
\end{eqnarray}
Then it holds from (\ref{lkpqeq2}), (\ref{lkpqeq3}) and (\ref{lkpqeq4}) that 
\begin{eqnarray}\label{lkpqeq5}
&& \sum_{\gamma\in \Gamma_{p}(K_{n})}
\bigg(
\sum_{\gamma'\in \Gamma_{q}(G_{\gamma})}{\rm lk}(f(\gamma),f(\gamma'))^{2}
\bigg)\\
&=&
(p-3)!(q-3)!\binom{n-6}{p-3}\sum_{\gamma'\in \Gamma_{3}(K_{n})}
\bigg(
\sum_{\gamma\in \Gamma_{3}(G_{\gamma'})}{\rm lk}(f(\gamma),f(\gamma'))^{2}
\bigg)\nonumber\\
&=& 2(n-6)! \sum_{\lambda\in \Gamma_{3,3}(K_{n})}{\rm lk}(f(\lambda))^{2}.\nonumber
\end{eqnarray}
By (\ref{lkpqeq}), (\ref{lkpqeq5}) and $2/(1 + \delta_{pq}) = 2-\delta_{pq}$, we have the result. 
\end{proof}


%
{\normalsize
}


\begin{thebibliography}{99}

\bibitem{CG83}
J. H. Conway and C. McA. Gordon, 
Knots and links in spatial graphs, 
{\it J. Graph Theory} {\bf 7} (1983), 445--453.





\bibitem{KK14}
A. A. Kazakov and Ph. G. Korablev, 
Triviality of the Conway--Gordon function $\omega_{2}$ for spatial complete graphs, 
{\it J. Math. Sci. (N.Y.)} {\bf 203} (2014), 490--498. 



\bibitem{MN19}
H. Morishita and R. Nikkuni, 
Generalizations of the Conway--Gordon theorems and intrinsic knotting on complete graphs, 
{\it J. Math. Soc. Japan} {\bf 71} (2019), 1223--1241. 

\bibitem{MN21}
H. Morishita and R. Nikkuni, 
Generalization of the Conway--Gordon theorem and intrinsic linking on complete graphs, {\it Ann. Comb.} {\bf 25} (2021), 439--470.


\bibitem{Nikkuni09}
R. Nikkuni, 
A refinement of the Conway--Gordon theorems, 
{\it Topology Appl.} {\bf 156} (2009), 2782--2794.


\bibitem{Nikkuni22}
R. Nikkuni, Topology of Spatial Graphs (regarding the Conway--Gordon theorems), SGC Library 178, SAIENSU-SHA Co.,Ltd., 2022. (in Japanese)


\bibitem{S84}
H. Sachs, 
On spatial representations of finite graphs, 
{\it Finite and infinite sets, Vol. I, II (Eger, 1981),} 649--662, 
Colloq. Math. Soc. Janos Bolyai, {\bf 37}, {\it North-Holland, Amsterdam,} 1984.


\bibitem{YL10}
A. Yu. Vesnin and  A. V. Litvintseva, 
On linking of hamiltonian pairs of cycles in spatial graphs (in Russian), 
{\it Sib. \`{E}lektron. Mat. Izv.} {\bf 7} (2010), 383--393




\end{thebibliography}
\end{document}